\documentclass[11pt,a4paper]{article}
\usepackage[latin1]{inputenc}
\usepackage{amstext}
\usepackage{graphicx}
\usepackage{amsfonts}
\usepackage{latexsym}
\usepackage{amssymb}
\usepackage{amsthm}
\usepackage{amsmath}
\usepackage[margin=4cm,bottom=4cm,top=4cm]{geometry}

\title{Phase Transitions in Partially Structured Random Graphs}
\author{Oskar Sandberg\thanks{The Department of Mathematical Sciences, Chalmers
    University of Technology and Göteborg
    University. ossa@math.chalmers.se.}}
\date{\today}

\newtheorem{theorem}{Theorem}[section]
\newtheorem{lemma}[theorem]{Lemma}
\newtheorem{definition}[theorem]{Definition}
\newtheorem{proposition}[theorem]{Proposition}

\newtheorem{conjecture}[theorem]{Conjecture}

\newcommand{\PR}{\mathbf{P}}

\newcommand{\E}{\mathbf{E}}
\newcommand{\Var}{\mathbf{Var}}

\newcommand{\R}{\mathbb{R}}
\newcommand{\Z}{\mathbb{Z}}
\newcommand{\N}{\mathbb{N}}

\newcommand{\convP}{\overset{p}{\rightarrow}}

\newcommand{\links}{\leftrightarrow}

\begin{document}

\maketitle

\begin{abstract}
We study a one parameter family of random graph models that spans a
continuum between traditional random graphs of the Erd\H{o}s-R\'enyi
type, where there is no underlying structure, and percolation models,
where the possible edges are dictated exactly by a geometry. We find
that previously developed theories in the fields of random graphs and
percolation have, starting from different directions, covered almost
all the models described by our family. In particular, the existence
or not of a phase transition where a giant cluster arises has been
proved for all values of the parameter but one. We prove that the
single remaining case behaves like a random graph and has a single
linearly sized cluster when the expected vertex degree is greater than
one.
\end{abstract}

\section{Random Graphs and Finite Percolation}


The $\mathcal{G}(n,p)$ family of random graphs, originally due to
Erd\H{o}s and R\'enyi, is constructed by letting the vertex set $V =
[n] = \{0,1,\hdots,n-1\}$, and letting every possible edge $\{u,v\}$
belong to the edge set $E$ independently with probability $p$.
Bernoulli bond percolation models, on the other hand, are typically
constructed by starting with a finite degree lattice and retaining
only the edges therein in the same manner, and not any others
(retained edges are called open).

The study of random graphs is thus the study of models without an
underlying geometry, whereas percolation models depend heavily on the
geometry and structure of the lattice on which they are defined. A
reasonable question is to ask what happens if one relaxes the
influence of the structure in percolation models -- for example by
allowing open edges to form between vertices more then one step from
each other in the lattice. Such models, known as \emph{long-range
  percolation}, have been studied previously (see
\cite{newman:percolation} \cite{aizenman:discontinuity} and below for
more references).

Likewise, starting from the other direction, one may ask what happens
to $\mathcal{G}(n,p)$ like graphs when structure is introduced --
making some possible edges more likely to appear than others. Results
about this can be gleaned from recent generalized random graph models
\cite{bollobas:phase}, and show how much structure can be introduced
while keeping the behavior of the model more or less intact.

In fact, in terms of the distance dependence of the edges, known
models for long-range percolation and generalized random graphs come
very close to covering the whole spectrum.  We will discuss a family
of random graph models with a single real parameter $\alpha$ that
regulates the influence of an underlying structure. We will see that
the cases when $\alpha < 1$ fall in the category of previously
analyzed random graphs, while the cases where $\alpha > 1$ fall in the
category of long-range percolation models. We present some
connectivity results for the final, critical, case where $\alpha = 1$.

\subsection{Notation}

As is common, we will use $G$ to denote both specific graph
realizations and the random graphs, though we strive to make the
difference clear by context. Where $\mathcal{G}$ is a random graph
family, $G \sim \mathcal{G}$ means that $G$ is distributed according to
this family. $C_1 = C_1(G)$ will denote the biggest connected
component of $G$.

As usual, a series of events $A_1, A_2, \hdots$ occurring
\emph{asymptotically almost surely} (a.a.s.) means that
$\lim_{n\rightarrow \infty}\PR(A_n) = 1$.

\subsection{Organization}

The paper is organized as follows. In Section \ref{sec:alpha} we
introduce our ``$\alpha$-model'' of random graphs, and in Section
\ref{sec:regimes} we go through how the behavior of the model varies,
discussing the known results for most values of $\alpha$. Finally
Section \ref{sec:analysis} contains our analytic contribution.

\section{The $\alpha$-model}
\label{sec:alpha}

A major difference between the analysis of random graphs and
percolation models is whether it is done in a finite or infinite
setting.  Questions about percolation are typically asked about the
behavior of clusters on a infinite grid -- the most basic question
being whether an infinite cluster remains open. $\mathcal{G}(n,p)$ is,
on the other hand, almost always studied for finite values of $n$ -
this for the simple reason that if $n = \infty$ and $p > 0$ the graph
is a.s.  not locally finite. The typical approach is instead to scale
the value of $p$ with $n$ -- in particular, sparse random graphs are
ones where $p = c/n$ for some fixed $c$ (meaning that expected degree
is essentially constant for all $n$). The question is then, rather
than asking whether an infinite cluster exists, to look at the
relative size of the largest cluster (compared to $n$) as a function
of $c$.

We take the latter approach here, using a degree normalizer and
studying finite graphs, but note that it largely intersects with
normalizer free models in cases where the degree is already limited by
the structure.  We will also restrict ourselves to a one dimensional
geometry. One dimension is not an interesting environment for standard
Bernoulli percolation, but long-range percolation can be fruitful
here. Our geometry is based around the following metric:
\[
d(u, v) = \min(|u-v|, n - |u - v|).
\]
This is equivalent to placing the vertices in a ring and using the
geodesic distance (see Figure \ref{fig:graph}).

\begin{definition} \emph{(The $\alpha$-model)} For $\alpha \in [0,
  \infty]$ the family of random graphs $\mathcal{G}_\alpha(n, c)$ are
  graphs $G = (V, E)$, where $V = [n]$ and for $u, v
  \in V$
\[
p_{u,v} = \PR(\{u,v\} \in E) = {c \over h_{\alpha, n} d(u,v)^\alpha}
\]
where $h_{\alpha, n} = \sum_{u \in V : u \neq 0} 1/d(u,0)^\alpha$ ,
independently for all disjoint $\{u, v\}$.
\label{def:alphamodel}
\end{definition}

For $\alpha = 0$ this equivalent to $G(n, p)$ with $p = c/(n-1)$. When
$\alpha = \infty$
\[
p_{uv} =
\begin{cases}
  c/2 & \text{ if } d(u,v) = 1 \\
0 & \text{ otherwise.}
\end{cases}
\]
which is standard percolation, for which an infinite cluster cannot
exist if $c < 2$.

Random graphs with edge probabilities given by a power-law of the
distance are not new, and have appeared in more or less exactly this
form elsewhere. See the pioneering work of Aizenman, Newman, and
Schulman \cite{newman:percolation} \cite{aizenman:discontinuity}, the
ideas of Kleinberg \cite{kleinberg:smallworld}, as well as later work
by other authors \cite{benjamini:diameter},
\cite{coppersmith:diameter}.

\begin{figure}
  \begin{center}
    \includegraphics[width=5.0cm]{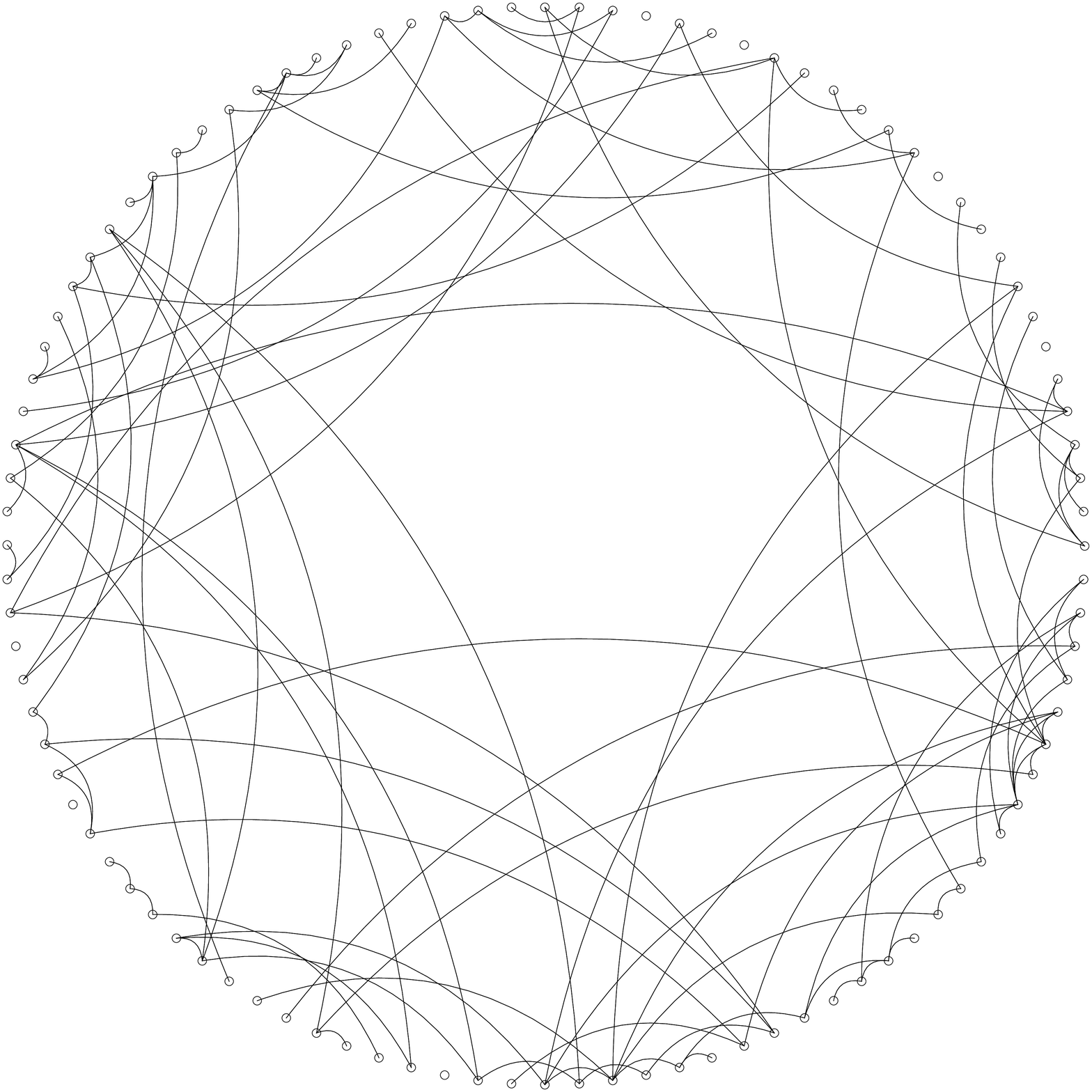}
  \end{center}
  \caption{A realization of the $\alpha$-model $\mathcal{G}_\alpha(n, c)$ with
    $\alpha = 1$, $n = 100$, and $c = 2$.}
  \label{fig:graph}
\end{figure}

\section{Regimes of the $\alpha$-model : The Emergence of Structure}
\label{sec:regimes}
\subsection{$\alpha = 0$ : $\mathcal{G}(n,p)$ Random Graph}
\label{sec:alpha0}
When $\alpha = 0$, the distance between the points does not affect
connectivity, and the $\alpha$-model is exactly the same as
$\mathcal{G}(n,p)$ with $p = c/(n-1)$.  This case has, of course, been
extensively studied, see \cite{erdos:randomgraphs} as well as the book
length discussions in \cite{janson:random} \cite{bollobas:random}.
With regard to connectivity, it is known to undergo a phase transition
at $c = 1$: the largest connected cluster is of size $\theta(\log n)$
in the subcritical phase, $\theta(n^{2/3})$ in the critical phase, and
$\theta(n)$ in the supercritical phase.

\begin{theorem}\emph{(Erd\H{o}s, R\'enyi)}
  Let $G \sim \mathcal{G}(n,p)$ for $p = c/(n-1)$, and $\rho$ be the
  survival probability of a Galton-Watson branching process with
  Poisson$(c)$ offspring distribution. Then
\[
|C_1|/n \convP \rho
\]
 as  $n \rightarrow \infty$.
\label{th:ergraph}
\end{theorem}

Much more is known regarding the distribution sequence of component
sizes, see the above books for details.

\subsection{$0 < \alpha < 1$ : Essentially a Random Graph}
\label{sec:alpha01}

When $\alpha > 0$ the geographic structure of the model starts
affecting the edges. However, as long as $\alpha < 1$ much of the
general behavior is retained.  While the random graph results above
cannot be directly applied, this regime falls within a more recent
general random graph model of Bollob\'as, Janson, and Riordan
\cite{bollobas:phase}.

In the BJR model $\mathcal{G}(n, \kappa)$, one is given a ``ground
space'' $\mathcal{S}$ which we take to be $[0,1]$, and a Borel
measurable kernel $\kappa : S \times S \mapsto \R^+$. For a sequence
of points $(x_1, x_2, \hdots, x_n)$ of $\mathcal{S}$ corresponding the
$n$ vertices, edges are added independently between each pair with
probability:
\[
p_{uv} = \min(\kappa(x_u, x_v) / n, 1).
\]
The model contains a lot of freedoms which we will not require. For
instance the sequence of points may be random, in which case the
limiting distribution obviously matters greatly. We will not need
this, and indeed may set $x_u = u/n$. We then let:
\begin{equation}
\kappa(x, y) =
\begin{cases}
  c / |x - y|^{\alpha} & \text{ for } x \neq y \\
0 & \text{otherwise}
\end{cases}
  \label{eq:bjrkern}  
\end{equation}
where $|x - y|$ is again interpreted as circular distance, this time
on $[0,1]$.

In order for the results of BJR to hold, $\kappa$ must adhere to
certain conditions which the authors call being ``graphical''
(Definition 2.7 in \cite{bollobas:phase}). The troublesome condition
for our $\kappa$ is that it must belong to $L^1(S \times S, \mu \times
\mu)$ (where $\mu$ in our case is the Lebesque measure on
$[0,1]$). Clearly this is true only for $\alpha < 1$.

For this case however, the BJR model with the above kernel gives
$p_{ij} = c / (n^{1 - \alpha} d(i, j)^\alpha)$ which is asymptotically
equivalent to the $\alpha$ model but for a slightly different value of
$c$. BJR prove that $\mathcal{G}(n, \kappa)$ behaves more or less like the
classical $\mathcal{G}(n, p)$ -- in particular, most of the proofs are very
similar barring the technical difficulties incurred by the greater
generality. With regard to connectivity, the same result as above
holds, that the phase transition at which a $\theta(n)$ cluster
emerges is at $c = 1$. 

\begin{theorem}\emph{(Bollob\'as, Janson, Riordan)}
  If $G$ is a random graph of size $n$ from the BJR model with kernel
  $\kappa$, then
\[
|C_1| / n \convP \rho_{\kappa}
\]
where $\rho_\kappa$ is the survival probability of a multitype
Galton-Watson process with the offspring distribution of $x$ given by
a Poisson point process on $\mathcal{S}$ with intensity $\kappa(x, \cdot)$.
\label{th:bjrgraph}
\end{theorem}

Because our geometries are transitive, this reduces to a single-type
Galton-Watson process with Poisson offspring. It is thus well known
that $\rho > 0$ exactly when the expected number of offspring is
greater than 1.

\subsection{$\alpha = 1$ : The Small World}
\label{sec:alpha1}

When $\alpha = 1$, the resulting model no longer falls within the that
of BJR since the $\kappa$ implied by (\ref{eq:bjrkern}) is no longer
integrable at 0. On the other hand, it does not fall within the
long-range percolation models discussed below, this time since $1/x$
is not integrable at infinity.  The connectivity of the resulting
``in-between'' model has to our knowledge not been studied elsewhere.

This model is of added interest due to the results of Kleinberg
regarding small-world models and further work that has followed, see
\cite{kleinberg:smallworld} \cite{kleinberg:networks}. These show that
the value of $\alpha$ is intimately related with possibility of
decentralized routing (path-finding) in graphs, and that it is exactly
when the relation between distance and edge prevalence is as at
$\alpha = 1$ that such routing will find short paths.

In Section \ref{sec:analysis} below, we prove that this case also
undergoes at phase transition equivalent to the previous -- when the
expected degree of each vertex is greater than one, $C_1$ has order
$n$, otherwise it is sublinear.

\subsection{$1 < \alpha < 2$: Long-Range Percolation}
\label{sec:alpha12}

When we move to $\alpha > 1$ the model undergoes another radical
change. In this case the normalizing constant $h_{\alpha, n}
\rightarrow h_{\alpha} < \infty$ as $n \rightarrow \infty$. The model
is thus very similar to the long-range percolation models studied in
the 1980s within mathematical physics (see \cite{newman:percolation}
and \cite{aizenman:discontinuity} for rigorous results). It was found
that these models do undergo a phase transition similar to the regimes
above: for $c$ less than a certain value there is no percolation,
while for large $c$ it can be shown to occur.

While the theorem is stated for percolation on $\mathbb{Z}$ in
\cite{newman:percolation}, the arguments in the proofs use only
finite subsets, and may be stated as

\begin{theorem} (Newman, Schulman) For $1 < \alpha < 2$ there exists
  $c < \infty$ and $\phi > 0$ such that for $\mathcal{G}_\alpha(n, c)$ 
\[
|C_1| \geq \phi\,n\;\;\text{ a.a.s.}
\]
\end{theorem}

In all the regimes of the $\alpha$-model discussed in Sections
\ref{sec:alpha0} -- \ref{sec:alpha1}, the critical value of $c$ was
found to be one. This is not possible when $\alpha > 1$.

\begin{proposition}
  For any $\alpha > 1$, there exists an $\epsilon > 0$ such that for
  $c = 1 + \epsilon$ in $\mathcal{G}_\alpha(n, c)$ $|C_1|/n \rightarrow 0$
  a.a.s.
\label{prop:cbig}
\end{proposition}

\begin{proof}
  Let $\epsilon = 1 / (6 h_{\alpha}^3)$. For a given graph, let $K(v)$ be the
  clustering number of a vertex $v$. That is
\[
K(v) = \text{number of 3-cycles containing }v
\]
For a given vertex $v$, let $Y$ be the number of second order neighbors.
\[
Y \leq \sum_{u \in N(v)} (|N(v)| - 1) - K(v)
\]
which implies
\[
\E[Y] \leq \E[|N(v)|]\E[|N(u)|] - \E[K(v)].
\]
Since $\E[K(v)] \geq \PR(v \links v+1,v \links v+2, v+1 \links v) = 1 /
(2 h_{\alpha} k^3)$, it holds that $\E[Y] < 1$. It follows that any two-step
exploration process on the graph is dominated by a subcriticial
Galton-Watson process.
\end{proof}

\subsection{$\alpha = 2$ : The Second Critical Value}
\label{sec:alpha2}

The situation when $\alpha = 2$ is also handled in
\cite{newman:percolation} and \cite{aizenman:discontinuity}.  Here it
turns out that for the $\alpha$-model, as we have defined it, there is
no giant cluster unless $c$ is large enough that $p_{u,u+1} = 1$ (in
which case everything is of course trivially connected). However, the
authors show that there are in fact distributions with this tail decay
where you will get a giant cluster, however, it is necessary that:
\[
\liminf_{d \rightarrow \infty} d^\alpha p_{u, u + d} \geq 1
\]
In our case, this implies that $c = h_{\alpha}$, which of course is the trivial
case where $p_{i,i+1} = 1$. For more general long-range percolation
formulations one can define non-trivial situations which do percolate
with this tail.

\subsection{$2 < \alpha < \infty$ : Essentially Percolation}
\label{sec:alpha2infty}

When $\alpha > 2$, the model shows behavior similar to standard
percolation. A way to see this is to consider a renormalization of the
vertex space into large blocks. For $c < h_{\alpha}$ let $B_1, B_2,
\hdots, B_{n/m}$ be contiguous blocks of vertices, each of size $m$.
Let $X_{i,j}$ be the number of edges between blocks $i$ and $j$.
\[
\E[X_{i,i+1}] \leq \sum_{x = -\infty}^0 \sum_{y = 1}^\infty {1 \over h_{\alpha}
  |x - y|^\alpha} \leq \infty.
\]
Since $X_{i,i+1}$ is the sum of independent events each occurring with
probability smaller than one, that it has bounded expectation means
that 
\[
\PR(X_{i, i+1} > 0) \leq p < 1
\]
where $p$ is independent of $n$ and $m$. On the other hand, if $|i -
j| > 1$, then
\[
\PR(X_{i,j} > 0) \leq
\E[X_{i,j}] < h' / m^{\alpha - 2} \rightarrow 0
\]
as $m \rightarrow
\infty$ ($h'$ is a constant).

This means that if $m$ is sufficiently large, and we view two blocks
as connected if there is any edge between them, then the system of
connected blocks will look more or less like standard Bernoulli
percolation.

Formally, let $k = \log n$ and $m = (\log n)^{3 / (\alpha - 2)}$. Then
the integral bounds give that for any vertex $x$
\begin{eqnarray*}
 & & \PR(x\text{ is connected to at least }mk\text{ vertices}) \leq \\
 & & \;\;\;\; \PR(k\text{ adjacent blocks are connected in either
   direction from } x) \\
 & & \;\;\;\; +\; \PR(\text{one of those } 2k \text{ blocks has a non-adjacent
  connection}) \leq 1 / \log n.
\end{eqnarray*}
for $n$ sufficiently large. It follows that $\E[\text{\# vertices in
  clusters larger than } mk] < n / \log n$ whence $\PR(|C_1|/n > \rho)
\rightarrow 0$ for any $\rho > 0$.


\subsection{$\alpha = \infty$ : Percolation}

As noted above, at $\alpha = \infty$ the $\alpha$-model is exactly
percolation with probability $c/2$ that each edge is open.

\section{Analysis of case $\alpha = 1$}
\label{sec:analysis}

In this section, we prove a result similar to Theorems
\ref{th:ergraph} and \ref{th:bjrgraph} for the ``Small World'' case
where $\alpha = 1$. 

To $\mathcal{G}_1(n,c)$ we associate a Galton-Watson branching process
$\{Z^n_t(c)\}_{t \in \Z^+}$ where the child distribution is the same
as the marginal distribution of each vertices degree in
$\mathcal{G}_{1}(n,c)$. Let $\rho_n = \rho_n(c)$ be the survival
probability of this process, which we know tends to 0 if $c \leq 1$
and a positive value otherwise.

By the ``law of rare events'', the distribution $\delta_i$ converges
to a Poisson$(c)$ distribution. Before proceeding we prove that this
implies that $\rho_n(c) \rightarrow \rho(c)$ as $n \rightarrow
\infty$, where $\rho(c)$ is the survival probability of GW process
with Poisson$(c)$ offspring. This continuity result isn't
new\footnotemark{} and seems to be assumed by some works on random
graphs, but since we need it explicitly at several points below, we
include it here.

\begin{lemma}
  Let for each $k \in \N$, $\{Z^k_i\}_{i = 0}^{\infty}$ be a
  Galton-Watson branching process with offspring given by a
  non-degenerate distribution with probability generating function
  $f_k$ and extinction probability $q_k$ (so that $q_k=f(q_k)$).

  If $f_k \rightarrow f$ as $n \rightarrow \infty$ pointwise, then
  $q_k \rightarrow q$, the smallest value in $[0,1]$ for which $f(q) =
  q$.
\label{lm:gw}
\end{lemma}

\begin{proof}
  Recall that the probability generating functions involved are
  convex, and take the value 1 at 1. Choose a subsequence $k_i$ such
  that $q_{k_i} \rightarrow \bar{q}$.
\[
\bar{q} \leftarrow q_{k_i} = f_{k_i}(q_{k_i}) \rightarrow f(\bar{q})
\]
all as $n \rightarrow \infty$. It follows that $\bar{q}$ is a fixed
point of $f$. If $q = 1$ then that is $f$'s only fixpoint in $[0,1]$,
and it follows directly that $\bar{q} = q$ and the result is
established. If $q < 1$, $f$ now has two fixpoints in $[0,1]$ by
convexity: $1$ and $q$. We must rule out the case $\bar{q} = 1$.

Assume that $q_{k_i} \rightarrow 1$. Let $q < s < q_{k_i}$, which
implies that $f_{k_i}(s) > s$. Letting $i \rightarrow \infty$ for all
$q < s < 1$, $f(s) \geq s$. But then there cannot exists a fixpoint $q
< 1$ such that $f(q) = q$, which contradicts out assumption.
\end{proof}

\footnotetext{Thanks to Peter Jagers for helping me with this proof.}

\begin{theorem}
  If $G \sim \mathcal{G}_{1}(n,c)$ then
\[
{|C_1(G)| \over n} \convP \rho \;\text{ as } n \rightarrow \infty
\]
\label{th:main}
\end{theorem}

The proof largely follows the proof for Erd\H{o}s-R\'enyi type graphs (see
$\cite{janson:random}$, and also $\cite{bollobas:phase}$.  We follow
the proof of Lemma 9.6 in the latter without significant deviation up
to the proof of Claim 2 in the latter half).  The difference is that,
because of the clustering, the branching process coupling breaks down
sooner. Therefore we need a different argument for why all ``large''
clusters are in fact the same.

\begin{proof}
  Choose $c'$ and $\epsilon$ such that $1 < c' < c$ and $0 <
  \epsilon < 1 - 1/c'$. 

We construct the graph using the following well known coupling: For
every pair of vertices $u$ and $v$, let $U_{u,v}$ be a random variable
uniformly distributed on $[0,1]$. We add an edge between $u$ and $v$
if $U_{u,v} < p_{u,v}$.

Let $G \sim \mathcal{G}_1(n,c')$ constructed in this manner. Later we will
increase the $p_{u,v}$ by
\begin{equation}
\delta \over h_{1,n} d(u,v)
\label{eq:exed}
\end{equation}
where $\delta = c - c'$. The resulting graph is distributed the same
as $\mathcal{G}_1(n,c)$.

Note that 
\begin{equation}
\log(n) \leq h_{1, n} \leq 2 \log(n)
  \label{eq:hbounds}
\end{equation}

Consider a standard exploration process on $G$ starting at a vertex
$x$. We terminate the exploration either when the explored set becomes
larger than some function $\omega(n) \leq n^{\epsilon}$ where
$\omega(n) \rightarrow \infty$ as $n \rightarrow \infty$ or when the
exploration process dies. Following \cite{bollobas:phase} (but with a
modified definition) we call such functions \emph{admissible} if for
any $\gamma > 0$
\[
\omega(n) / n^{\gamma} \rightarrow 0 
\]
as $n \rightarrow \infty$. Let the set $B = B_{\omega}$ be set of $x$
for which the process stopped for the former reason ($B$ thus contains
all the vertices in components larger than $\omega(n)$).

Since $\omega(n) \ll n^{\epsilon}$, the exploration may be coupled
between two branching processes. From above, we can couple it with
$Z^n(c')$, and from below by a similar process but where the offspring
are given a random variable $Y$, the degree of $x$ only counting
neighbors more than $\omega(n)$ steps away (since the worst case is
that we have already explored the $\omega(n)$ nearest vertices). Using
(\ref{eq:hbounds}) this gives
\[
\E[Y] \geq c' \left (1 - {\log \omega(n) \over \log n} \right ) \geq c' (1-\epsilon).
\]
Let $\rho' = \rho'(c')$ be the survival probability of this process. It then
follows that, 
\[
\rho' \leq \PR(x \in B) \leq \rho(c') + o(1).
\]
By selecting $\epsilon$ sufficiently small, we can make $\rho'$
arbitrarily close to $\rho(c')$ (Lemma \ref{lm:gw}), while the
coupling still holds for $n$ sufficiently large. Thus $\PR(x \in B)
\rightarrow \rho(c')$ as $n \rightarrow \infty$. Addition over all the
vertices gives
\begin{equation}
{1 \over n} \E|B| \rightarrow \rho(c').
  \label{eq:expconv}
\end{equation}

What remains is to show two things:
\begin{enumerate}
\item For all admissible $\omega(n)$, $|B| / n \convP \rho(c')$ as $n \rightarrow \infty$.
\item For some admissible $\omega(n)$, $B$ consists of only one component.
\end{enumerate}

To prove the first claim, we note that the derivation of
(\ref{eq:expconv}) did not depend on the choice of $\omega(n)$, and
thus holds for all admissible functions. Given such a function,
let $\omega'(n)$ be one strictly larger, and let $B$ and $B'$ be their
respective sets of vertices in large components (note that $B' \subset
B$). 
\begin{eqnarray}
{\E|B \backslash B'| \over n} = {\E|B| - \E|B'| \over n} \rightarrow 0.
\label{eq:ebb}
\end{eqnarray}
It follows that if $|B| / n \convP \rho(c')$ holds for $B$, it must also hold
for $B'$, since
\begin{eqnarray*}
\PR\left ( \left | {|B'|\over n } - \rho \right | > \epsilon_1
\right ) & = & \PR\left ( \left | {|B| - |B \backslash B'| \over n } - \rho
  \right | > \epsilon_1 \right ) \\
& \leq & \PR\left ( \left | {|B|\over n } - \rho \right | > \epsilon_1/2
\right ) \\
& & \; + \;\PR\left ( {|B \backslash B'|\over n } > \epsilon_1/2
\right) \rightarrow 0
\end{eqnarray*}
The second term of the convergence follows by (\ref{eq:ebb}) and the
first moment method.

Now let $\omega(n) \leq \log \log(n)$. We will show the claim for this
$\omega(n)$, and use the previous result to establish it for any
faster growing $\omega(n)$. We now explore from two vertices $x$ and
$y$.  Start the exploration from $x$ first. At the end of this, we
have found a connected subset $C(x)$ of vertices around $x$. Because
both the expected degree of each vertex and the variance is constant,
a Chebyshev bound shows that the probability that we should encounter
a vertex with more than $\omega(n)$ neighbors is $o(1)$. Thus 
we can assume that $|C(x)| \leq 2 \omega(n)$.
\[
\PR(y \in C(x)) \leq {2 \omega(n) \over \log n} = o(1)
\]
since at each step of the exploration, the probability that any vertex
which is not $y$ is connected to it is less than $1 / 2 \log n$ by
(\ref{eq:hbounds}). Next we explore from $y$ and until we have
constructed a $C(y)$. In each step of the exploration, the probability
we draw a vertex in $C(x)$ next is bounded from above by $\log \log
\log n / \log n$, so
\[
\PR(C(x) \cap C(y) \ne \emptyset) \leq 2 \log \log n {\log \log \log
  n\over \log n} + o(1) = o(1).
\]
It follows that:
\[
\rho' \rho' - o(1) \leq P(x, y \in B) \leq \rho(c') \rho(c') + o(1)
\]
whence $P(x,y \in B) = \rho^2(c')$ and 
\[
{1 \over n^2} \E[\,|B|^2] \rightarrow \rho^2
\]
as $n \rightarrow \infty$. This means that $\Var(|B| / n) \rightarrow
0$ and thus $|B| / n \convP \rho(c')$. The first claim is thus
established.

For the second claim, we will add the additional edges that we
withheld in the beginning by increasing the threshold for edge
existence by (\ref{eq:exed}), and show that this connects all large
clusters. We start by letting $\omega(n) = \log^4(n)$ and
$B$ be as before. We condition on the graph $G$ constructed with the
$c'$ threshold, which we may assume has $|B| \geq (\rho - \epsilon_2) n$
(for $\epsilon_2$ arbitrarily small) by the above. From $G$ we create
the graph $G^c$ by completing all the connected clusters of $G$ - note
that while this adds edges, it does not change the connectivity
properties of the graph. In particular, if adding the additional edges
makes $B$ a connected component in $G^c$, it does so also in $G$.


Now select from $B$ as many subsets $K_1, K_2, \hdots, K_m$ as
possible, such that each $K_i$ is a clique in $G^c$, and each $|K_i| =
\log^3(n)$. Since $B$ consists only of connected clusters of size at
least $\log^4(n)$ in $G^c$, we can select these so that $m = (\rho -
\epsilon_2 - o(1)) n / \log^3(n)$.

Consider now the graph $H$, created by taking the $K_i$ as vertices,
as connecting $K_i$ and $K_j$ if a new edge is created between any two
constituent vertices when the $\delta n$ edges are added. Since for
any vertices $x$ and $y$, $d(x,y) \leq n/2$
\[
\PR(\mbox{$x$ and $y$ are connected by the new edges}) \geq {\delta
  \over 2 n \log n}
\]
It follows that the number of connections created between $K_i$ and
$K_j$ dominates a random variable $X$ which is $\text{Bin}(\log^6(n),
\delta / 2 n \log n)$ distributed. From a simple second moment
estimate, one gets
\[
\PR(X > 0) \geq {\delta \log^6(n) \over 4 n \log n} = \left
  (\delta \log^2(n) \over 4 \right ) {\log^3(n) \over n}.
\]
It follows that $H$ is dominated by a graph of the form
\[
\mathcal{G} \left ((1 - \epsilon_2 - o(1)) {n \over \log^3(n)}, \left
  (\delta \log^2(n) \over 4 \right ) {\log^3(n) \over n} \right )
\]
of the standard Erd\H{o}s-R\'enyi $\mathcal{G}(n,p)$ family. But the
threshold for $\mathcal{G}(n,p)$ being completed connected a.a.s.\ is
$p \gg \log n / n$, which holds here. Thus $H$ is a.a.s.\ connected,
from which it follows that $B$ is a.a.s.\ connected in the completed
graph. This establishes the result.
\end{proof}

\section{Generalisation and a Conjecture} 

The result in Section \ref{sec:analysis} shows that in the
$\alpha$-model, $\alpha \leq 1$ gives a connectivity phase transition
similar to classical random graphs, while $\alpha > 1$ does not. It is
of interest to understand more precisely when this transition occurs.
That is, for more general $p_{u,v}$ when does having a mean vertex
degree greater than one sufficient for a giant component to emerge?
Put more philosophically: when is a random graph a random graph, and
when is it percolation?

A generalization of the $\alpha$-model in Definition
\ref{def:alphamodel} is given by letting
\begin{equation}
p_{u,v} = \PR(\{u,v\} \in E) = {c f(d(u,v)) \over h_{f, n}}
\label{eq:fmodel}
\end{equation}
where $f : \N \rightarrow \R$ is some decreasing function, and 
\begin{equation}
h_{f,n} = \sum_{u \in V : u \neq 0} f(d(0,u)).
\label{eq:hfn}
\end{equation}
$\mathcal{G}^\alpha(n,c)$ is thus given by letting $f(x) =
1/x^\alpha$. The question then becomes for which $f$ functions $c = 1$
is the critical value above which a giant component emerges.

Proposition \ref{prop:cbig} directly provides a necessary requirement,
namely that $h_{f,n} \rightarrow \infty$ as $n \rightarrow \infty$.
Thus for instance $f(x) = 1 / (x \log^{1 + \epsilon} x)$ for $\epsilon
> 0$ cannot behave in the ``random graph manner''. On the other hand,
by retracing the steps of the Proof of Theorem \ref{th:main}, for the
case $f(x) = 1 / (x \log x)$, where $h_{f,n} \approx \log\log n$, it
is easy to see that the critical $c$ is exactly 1 also in this case.

In fact, $f$ may be such that $h_{f,n}$ is any logarithm iteration,
and the methods of the proof still go through with little
modification. The only problematic requirement is that initial
coupling with a branching process should break down before we have
seen $\log^3 n$ vertices, which is the same as requiring that
\[
{ \sum_{x = 1}^{\log^3 n} f(x) \over \sum_{x=1}^{n} f(x) } \rightarrow 0
\]
as $n \rightarrow \infty$. This is of course true for all but the very
most slowly growing sums, but it does not hold for instance for $f$
such that $h_{f,n} = \log^* n$.

In light of this, not withstanding that proving it would require a
different method, we feel motivated to make the following conjecture
\begin{conjecture}
  In the random graph created by using edge probabilites given by
  equations (\ref{eq:fmodel}) and (\ref{eq:hfn}), $c = 1$ is the
  critical value for the emergence of a giant component if, and only
  if, $h_{f,n} \rightarrow \infty$ as $n \rightarrow \infty$.
\end{conjecture}

\section{Conclusion}

The $\alpha$-model spans the spectrum from structure-free random graph
models to ordinary percolation. When $\alpha \leq 1$, the connectivity
results more or less mirror those of random graphs, whereas for
greater values they behave more like percolation. 

We note that the cases where $\alpha \leq 1$ are exactly those where
$p_{u,v} \rightarrow 0$ as $n \rightarrow \infty$ for all $u \neq v$.
An interesting question is to further explore this territory and see
if this property, under some regularity (perhaps monotonicity)
requirements, is sufficient for ``random graph'' type behavior, or if
there are cases where this holds, but where the critical value is not
one.


\section*{Acknowledgments}
Thanks to my advisers Olle Häggström, and Devdatt Dubhashi, as well as
Oliver Riordan for some early discussion about the subject, and Peter
Jagers for some technical assistance. Also thanks to Svante Janson who
pointed out an error in an earlier version of this manuscript.

\bibliographystyle{unsrt}
\bibliography{../../tex/ossa}

\end{document}